\documentclass[12pt,a4paper]{amsart}
\usepackage{amssymb}
\usepackage{tikz-cd}
\usepackage{chngcntr, csquotes}
\usepackage[margin=0.75in]{geometry}

\usepackage{tikz-cd}
\usepackage{chngcntr, csquotes}
\usepackage{accents}
\usepackage{todonotes}
\usepackage{tikz}
\usepackage{hyperref}
\usepackage{cleveref}
\usepackage{enumitem}

\newcounter{proofcount}
\AtBeginEnvironment{proof}{\stepcounter{proofcount}}
\newtheorem{claim}{Claim}
\makeatletter                  
\@addtoreset{claim}{proofcount}
\makeatother                   

\theoremstyle{remark}
\makeatletter
\newtheorem*{cproof/}{Proof of Claim \rev@cproofmark}

\newenvironment{cproof}[1][\@nil]
  {\def\@tmp{#1}%
   \ifx\@tmp\@nnil
       \def\rev@cproofmark{\theclaim}
    \else
       \let\rev@cproofmark\@tmp
    \fi
   \pushQED{\qed}\begin{cproof/}}
  {\popQED\end{cproof/}}
\makeatother

\newcommand{\rr}{\mathbb{R}}
\newcommand{\zz}{\mathbb{Z}}

\newcommand{\origin}{O}
\newcommand{\del}[1]{\partial#1}
\newcommand{\spec}[1]{#1^{-}}

\theoremstyle{theorem}
\newtheorem{thm}{Theorem}[section]
\newtheorem{lem}[thm]{Lemma}
\newtheorem{cor}[thm]{Corollary}
\newtheorem{dfn}[thm]{Definition}
\newtheorem{prp}[thm]{Proposition}

\theoremstyle{definition}
\newtheorem{question}[thm]{Question}

\newtheorem*{rem}{Remark}

\numberwithin{equation}{section}

\DeclareMathOperator{\diam}{diam}

\DeclareMathOperator{\ran}{ran}

\newcommand{\hbor}{H^2_{\normalfont{\textsf{Bor}}}}

\newcommand{\cant}{2^{\omega}}
\newcommand{\baire}{\omega^{\omega}}

\newcommand{\olines}{\mathcal{L}_{\origin}}

\newcommand{\homeo}{\cong}

\newcommand{\ZFC}{\mathsf{ZFC}}

\newcommand{\pp}{\mathbb{P}}
\newcommand{\qq}{\mathbb{Q}}

\newcommand{\set}[2]{\{ \, #1 \,\mid\, #2 \, \}}
\newcommand{\Set}[2]{\left\{ #1 \,\middle\vert\, #2 \right\}}

\newcommand{\define}[1]{{\normalfont{\textbf{#1}}}}

\begin{document}

\title{All Borel Group Extensions of Finite-Dimensional Real Space Are Trivial}

\author{Linus Richter}
\address{Department of Mathematics, National University of Singapore, Singapore}
\email{richter@nus.edu.sg}

\subjclass{54H05 (Primary) 20K35, 20J06, 03E40 (Secondary)}
\keywords{Borel definability, descriptive set theory, group extensions,
group cohomology, forcing}


\begin{abstract}
For $N \geq 2$, we study the structure of definable abelian group extensions of the additive group~$(\rr^N,+)$ by countable abelian (Borel) groups~$G$. Given an extension $H$ of~$(\rr^N,+)$ by $G$, we measure the definability of $H$ by investigating its complexity as a Borel set. We do this by combining homological algebra and descriptive set theory, and hence study the Borel complexity of those functions inducing~$H$, the abelian cocycles. We prove that, for every~$N \geq 2$, there are no non-trivial Borel definable abelian cocycles coding group extensions of~$(\rr^N,+)$ by a countable abelian group~$G$, and hence show that no non-trivial such group extensions exist. This completes the picture first investigated by Kanovei and Reeken in 2000, who proved the case~$N = 1$, and whose techniques we adapt in this work.
\end{abstract}

\maketitle

\section{Introduction}\label{sec:intro}
In this paper, we contribute to the study of \emph{definable classical mathematics}. By classical mathematics we mean those mathematical subjects traditionally unrelated to mathematical logic. The idea of definability is interpreted broadly: as opposed to the universal yet rigid model-theoretical definition, we mean the classification of instances of classical mathematical problems in terms of measures at home in mathematical logic.

One of the earlier examples of the application of logic to classical mathematics was given by Shelah's celebrated solution to the Whitehead problem \cite{MR357114}, but many have followed suit.
Applications of this research programme feature prominently in major areas of contemporary logic such as reverse mathematics \cite{MR4472209,MR1723993} and computability theory \cite{MR882921}. More finely, the study of reductions and effective measures of complexity play an important role throughout modern logic, emphasised by notions such as Weihrauch reducibility \cite{Brattka2021}, computable analysis \cite{MR1795407}, algorithmic randomness \cite{MR2494387,MR2548883,MR2732288}, and many more. Crucially, many of these admit natural connections to classical mathematics, such as the effective relationship between fractal geometry and algorithmic randomness \cite{mayordomo02,lutz03,hitchcock03,MR3811993}. 

\medskip

In this paper, we continue this trend. We choose as our measure framework the study of subsets of the reals in terms of descriptive set theory and apply it to (homological) algebra: we consider the problem of group extensions of abelian groups by abelian groups, and measure which of them are \emph{definable} in the context of descriptive set theory. Analysing the \emph{Borel definability} of group extensions garnered interest in the 1960s through work of Moore \cite{MR171880} and DuPre \cite{MR233919}, and was re-investigated by Kanovei and Reeken in the early 2000s \cite{MR1774679,MR1841758,MR1896675}. Instead of taking a strongly algebraic point of view as Moore and DuPre had done, their work was motivated by ideas in \emph{Hyers-Ulam-Rassias-stability} \cite{MR4076,MR120127,MR280310,MR507327} which measures the stability of spaces (and, more generally, categories) in terms of homomorphisms: can every \emph{almost homomorphism} be closely approximated by a \emph{strict} homomorphism? For certain groups, group extensions stand in equivalence with maps called \emph{cocycles}, whose properties can be likened to Hyers-Ulam-Rassias-(in)stability. In combination with the measurability tools of Polish spaces, \emph{definable} group extensions can be carefully analysed.

\medskip

This analysis is our motivation for this paper. We generalise a result of Kanovei and Reeken concerning the existence of definable group extensions \cite{MR1841758} to higher-dimensional topological spaces, hence proving that the definable algebraic structure (in the sense of Borel measurability) of abelian group extensions of the additive groups $(\rr,+)$ and $(\rr^N,+)$ for $N > 1$ does not differ.

\medskip

To formalise this, let $A,G$ be groups. A \define{group extension} of $A$ by $G$ is a \emph{short exact sequence}: a sequence $c$ of the form
\begin{center}
\begin{tikzcd}
c \colon\hspace{-3em}  &0 \arrow[r] & G \arrow[r, "i"] & E \arrow[r, "j"] & A \arrow[r] & 0
\end{tikzcd}
\end{center}
where $i,j$ are group homomorphisms, $i$ is injective, $j$ is surjective, and where $\text{im}(i) = \ker(j)$. Two group extensions $c,c'$ with
\begin{center}
\begin{tikzcd}
c' \colon\hspace{-3em}  &0 \arrow[r] & G \arrow[r, "i'"] & E' \arrow[r, "j'"] & A \arrow[r] & 0
\end{tikzcd}
\end{center}
are \emph{equivalent} if there is a group isomorphism $p$ for which the following diagram commutes:
\begin{center}
\begin{tikzcd}
            &                                     & E \arrow[rd, "j"] \arrow[dd, "p"] &             &   \\
0 \arrow[r] & G \arrow[rd, "i'"'] \arrow[ru, "i"] & {} \arrow[r, phantom]             & A \arrow[r] & 0 \\
            &                                     & E' \arrow[ru, "j'"']              &             &  
\end{tikzcd}
\end{center}

If we require that all groups $G,E,A$ are abelian, then every extension of~$A$ by~$G$ is in fact induced by an \emph{abelian cocycle} (or \emph{factor set})---a function from~$A^2$ to~$G$ which encodes the structure of~$E$---and vice versa. This is how we study the structure of group extensions here: we consider group extensions from the point of view of abelian cocycles and their properties in terms of functional equations; their structure is captured by the \emph{cohomology~group}.

\medskip

The identification of group extensions in terms of functions between groups has spurred the investigation of definable group extensions: those induced by an abelian cocycle which is \emph{definable}. In this paper, \emph{definable} means \emph{Borel measurable}: we consider groups which carry a natural Polish structure, and investigate those group extensions which are induced by Borel cocycles.

\begin{dfn}\label{dfn:borelDfn}
	A group $G$ is a \define{Borel group} if $G$ is a Borel subset of a standard Borel (or Polish) space, and the group operation is a \define{Borel map}: the pre-image of every Borel set is Borel in the inherited topology.
	 
	Suppose $G,E,A$ are abelian Borel groups inside standard Borel spaces. A group extension $c \colon 0 \rightarrow G \rightarrow E \rightarrow A \rightarrow 0$ is \define{Borel definable} (or just \define{Borel}) if the associated abelian cocycle generating $E$ is a Borel map.
\end{dfn}

In order to take advantage of the identification between group extensions and cocycles, \textbf{we assume from now on that all groups are abelian}. Further, \textbf{the quotient $G$ is always assumed to be countable}.

\medskip

With these conventions established, the following theorem \cite[Theorem~49]{MR1841758} classifies---up to isomorphism between group extensions---all definable group extensions of the additive group $(\rr,+)$ by a (countable) group $G$ in the Borel framework.

\begin{thm}[Kanovei and Reeken, 2000]\label{thm:KRreals}
	For every group $G$, the only Borel definable group extension of $(\rr,+)$ by~$G$ is the trivial extension. In other words, every Borel cocycle $C \colon (\rr^2)^2 \rightarrow G$ is a Borel coboundary.
\end{thm}

In \cref{sec:thm1,sec:final}, we prove the following theorem, extending Kanovei and Reeken's result.

\newtheorem*{thm:thm2}{\Cref{thm:myThm2}}
\begin{thm:thm2}
For every $N < \omega$ and every group $G$, the only Borel definable group extension of~$(\rr^N,+)$ by~$G$ is the trivial extension.
\end{thm:thm2}

The proof uses geometrical arguments, which on the one hand are harder to prove in the multidimensional case as opposed to the case $\rr$ in \cite{MR1841758}, while on the other hand translate throughout the higher-dimensional cases. As general ideas of the proof of \cref{thm:myThm2} are most easily explained in the case $N=2$, we prove the following result first:

\newtheorem*{thm:thm1}{\Cref{thm:myThm1}}
\begin{thm:thm1}
For every group $G$, the only Borel definable group extension of $(\rr^2,+)$ by~$G$ is the trivial extension.
\end{thm:thm1}

\subsection{The structure of this paper}

In \cref{sec:background} we outline the basics of group theory and the theory of group extensions needed for our arguments. We also recall the fundamentals of descriptive set theory which we later use to classify the degree of definability of group extensions.

As we will be using forcing arguments, we shall also consider its basic notions. However, these will not concern extending models of set theory, and hence remain gentle.

In \cref{sec:thm1,sec:final}, we prove theorems \ref{thm:myThm1} and \ref{thm:myThm2}, respectively. The proof of \cref{thm:myThm1} follows structurally the analogous argument in \cite{MR1841758}, yet requires more sophisticated arguments to overcome the geometrical differences between $\rr$ and $\rr^2$. The lemmas we prove, however, can then be extended to prove the theorem on $(\rr^N,+)$ for any $N > 2$. We show how to do this in \cref{sec:final}, where we also conclude \cref{thm:myThm2}.

We close with \cref{sec:open}, in which we state various open questions, and where we outline related avenues for research.

\section{Background}\label{sec:background}
The theory of group extensions of abelian groups by abelian groups can be expressed in two ways: \emph{category-theoretically} (i.e.\ in terms of short exact sequences, as indicated in \cref{sec:intro}) and in terms of \emph{group cohomology}, which captures the relationship between certain functions between groups: \emph{cocycles} and \emph{coboundaries}. The former approach has lead to the study of the \textsf{Ext} functor \cite{MR3467030}, which is of course of independent interest. In this paper, however, we take the latter approach. We do this since: if all groups concerned have a Borel structure, then the Borel complexity of cocycles---which are themselves functions between groups---can be measured and hence analysed.

\medskip

Let $A$ be a Borel group. A function $C \colon A^2 \rightarrow G$ is an \define{abelian 2-cocycle} (or just \define{abelian cocycle}) if $C(x,0) = 0$ and
\begin{align*}
	C(x,y) = C(y,x) \text{ \; and \; }
	C(x,y) + C(x+y,z) = C(x,z) + C(x+z,y)
\end{align*}
for all $x,y,z \in A$. An abelian cocycle is a \define{coboundary} if there exists a map~$h \colon A \rightarrow G$ such that
\[
	C(x,y) = C_h(x,y) = h(x) + h(y) - h(x + y).
\]
The quotient group of cocycles modulo coboundaries yields the \define{cohomology group} $H^2(A,G)$ \cite[9.2]{MR2455920}, which measures complexity: the more complicated~$H^2(A,G)$ is as a group, the more complicated is the structure of group extensions of $A$ by $G$. It is clear that two cocycles are \define{equivalent} if their difference is a coboundary\footnote{To see the explicit relationship between cocycles $C \colon A^2 \rightarrow G$ and group extensions of~$A$ by $G$ we recommend in particular Chapter 9 of \cite{MR3467030}. There, it is also shown how the notion of equivalence given here agrees with the commutative-diagram-equivalence from \cref{sec:intro}.}.

\begin{rem}
The definition of coboundaries hints at the relationship between the theory of group extensions and the aforementioned Hyers-Ulam-Rassias-stability framework of almost homomorphisms: in the sense of Ulam, coboundaries are trivial objects (well-behaved almost homomorphisms), just as they are trivial objects in the study of group cohomology (by definition of the quotient group).
\end{rem}

To explain our theorems, we need to introduce descriptive set theory. A topological space is \define{Polish} if it is separable and completely metrisable. The \emph{definable} subsets of an uncountable Polish space $X$ can be stratified in hierarchies. At the lowest level sits the \define{Borel hierarchy}, whose sets are generated by the $\sigma$-algebra containing the open subsets of $X$. A set is \define{Borel} if it appears in this hierarchy of length $\omega_1$. A function $F \colon X \rightarrow Y$ between Polish spaces is a \define{Borel function} if the pre-image of every Borel set is Borel. Hence, every continuous function is Borel. Further, a function between Polish spaces is Borel if and only if its graph is \emph{analytic}, i.e.\ the continuous image of a Borel set \cite[II.14.12]{MR1321597}. Since Polish spaces are separable and Hausdorff, they have a countable basis of \define{basic open sets}, and so every Borel set $A$ has a \define{Borel code}: an element of~$\omega^\omega$ which codes the construction of $A$ from the basic open sets and countably many operations of countable union and complementation. For details, see \cite{MR1321597,MR2731169}.

A simple method to determine whether a set of reals is Borel is provided by the fundamental relationship between descriptive set theory and computability theory (see references \cite{MR2526093,MR2455198,MR3381097}). The following folklore result is emblematic of this relationship. Some basic concepts follow: if $\Phi$ is a computable procedure (e.g.\ a Turing machine) and $x$ is a real (in $\rr$, $\cant$, $\baire$\ldots), let $\Phi\left(x^{(\alpha)}\right)$ denote the procedure $\Phi$ initiated with oracle access to the $\alpha$-th jump of $x$. If $x,z \in \cant$ then $x \oplus z$ is the join of $x$ and $z$. (The join of elements of $\rr$ is usually expressed as the join of their binary expansions. For classical computability theory, see e.g.\ \cite{MR882921}.)

\begin{prp}[{cf.\ \cite{dansPaper,Richter2024}}]\label{prp:BorelDefn}
	Let $X$ be uncountable Polish. A set $A \subseteq X$ is Borel if there exists a code $z \in \baire$, an ordinal\footnote{To the reader familiar with hyperarithmetic theory: $\alpha < \omega_1^{z}$, the first ordinal not computable from $z$.} $\alpha < \omega_1$, and a computable procedure $\Phi$ for which
	\[
		x \in A \iff \Phi\left((x \oplus z)^{(\alpha)}\right) \text{ halts in finite time.}
	\]
\end{prp}

\begin{proof}
	To elucidate the intuition, we sketch the proof of the case $\alpha = 1$ for $X = \rr$. If $A$ is open, write $A = \bigcup (a_n,b_n)$ where $a_i,b_i \in \qq$. This sequence of pairs is coded by some $z$ (via e.g.\ the Cantor pairing function). For each pair $(a_n,b_n)$ coded in $z$, ask whether $a_n < x < b_n$. This is a finite operation since each interval $(a_n,b_n)$ is open.
\end{proof}

\subsection{Algebraic properties of abelian cocycles}\label{sec:identities}
In the subsequent arguments, we generally separate \emph{algebraic properties} of abelian cocycles from \emph{non-algebraic} properties.

As the name suggests, algebraic properties hold for all abelian cocycles; the underlying space and its topology are inconsequential. On the contrary, non-algebraic properties depend on the underlying topological and order-theoretical properties. For instance, the reals $\rr$ are ordered, while $\rr^2$ is not. Hence, any proof of \cref{thm:KRreals}, if adapted to $\rr^2$, must take care of such structural differences.

The difficulties of extending \cref{thm:KRreals} reduce to proving topological lemmas which permit the application of algebraic properties. This is our task in the following section. In this section, we collect useful algebraic properties of abelian cocycles---all can be found in \cite{MR1841758}. Fix an uncountable group $X$ and a group $G$ and let $C \colon X^2 \rightarrow G$ be an abelian cocycle.

\begin{lem}\label{algProp1}
	For all $x_1,x_2,x_3 \in X$ we have
	\[
		C(x_1,x_2) = C(x_1,x_3) + C(x_1 + x_3,x_2) - C(x_1 + x_2, x_3).
	\]
\end{lem}

Cocycles can be extended naturally to accept more arguments: let
\begin{align*}
	C(x_1,x_2,x_3) &= C(x_1,x_2) + C(x_1 + x_2,x_3)\\
\intertext{and hence recursively define}
	C(x_1,\ldots,x_{n+1}) &= C(x_1,\ldots,x_n) + C(x_1 + \ldots + x_n, x_{n+1}).
\end{align*}
Various useful identities can be deduced from this recursive definition, such as the following.

\begin{lem}
	If $x_1,\ldots,x_n,x_{n+1},\ldots,x_m \in X$, let $x = x_1 + \ldots + x_n$ and $x' = x_{n+1} + \ldots + x_m$. Then
	\[
		C(x_1,\ldots,x_m) = C(x_1,\ldots,x_n) + C(x_{n+1},\ldots,x_m) + C(x,x').
	\]
\end{lem}

Recall that a cocycle $C$ is a coboundary if there exists a map $h \colon X \rightarrow G$ for which
\[
	C(x_1,x_2) = h(x_1) + h(x_2) - h(x_1 + x_2).
\]
A simple induction proves:

\begin{lem}\label{alg1}
	$C_h(x_1,\ldots,x_n) = h(x_1) + \ldots + h(x_n) - h(x_1 + \ldots + x_{n})$
\end{lem}

Further, since $H^2(X,G)$ is a group it immediately follows that:

\begin{lem}\label{lastAlgPrp}
	If $C_h$ and $C_k$ are coboundaries, then $C_h + C_k = C_{h+k}$, where $(h+k)(x) = h(x) + k(x)$. In particular, if $C_h$ and $C_k$ are Borel coboundaries (i.e.\ $h$ and $k$ are Borel maps), then so is~$C_{h+k}$.
\end{lem}

\subsection{The set-theoretical setup}\label{subsec}

We work in a set-theoretical universe $V$ satisfying $\ZFC$. For a sufficiently rich finite fragment $\ZFC^*$ of $\ZFC$, consider a countable transitive set model~$M \vDash \ZFC^*$. Meta-theoretically, our arguments proceed as follows: with a Borel cocycle $C$ and $M$ fixed, we work inside $V$ and argue that $C$ is in fact a Borel coboundary. We often allude to the fact that truth in a sufficiently rich extension $M[G]$ of $M$ is forced by some condition (read open set) in $M$: there is a comeagre set \emph{in $V$} on which the function is constant, for instance. For suitable $G$, this gives us a result on $V$, hence proving that $C$ enjoys said properties \emph{in the real world}. The fact that Borel functions have codes whose interpretations are absolute between transitive models is crucial for this argument to work.

As is customary, when we write ``let $M$ be a countable transitive model of $\ZFC$'' we mean $M$ to satisfy a finite fragment of $\ZFC$ sufficiently large for the arguments at hand.

\medskip

In view of proving \cref{thm:myThm1}, we now consider the specific case $A = (\rr^2,+)$. We also write $\rr^2$ in place of $(\rr^2,+)$ to save time.

\medskip

While the following set-theoretical setup is expressed in terms of $\rr^2$, it should be clear that it in fact holds for $\rr^N$ for all $N < \omega$. We work in $V$ and fix a countable group $G$. Suppose $C \colon (\rr^2)^2 \rightarrow G$ is a Borel function. Let $z \in \baire$ be a Borel code for $C$. Now consider a fixed countable transitive model $M$ of~$\ZFC$ such that $\{ z,G \} \subset M$.

Let $\pp_2 = \set{(q,q') \times (r,r')}{q,q',r,r' \in \qq}$, our version of Cohen-forcing for pairs. As always, $p = I \times I' \leq q = J \times J'$ (i.e.\ $p$ is stronger than $q$) if and only if $I \subseteq J$ and $I' \subseteq J'$. Further, define $\pp$ to be the classical Cohen forcing whose conditions are open sets $(q,r)$ for $q,r \in \qq$. Then $\pp \times \pp = \pp_2$ as sets, and also as forcing notions, as the following classical theorem shows.

\begin{thm}[{\cite{MR756630}}]
	A filter $G_0 \times G_1$ is $M$-generic for $\pp_2$ iff one of the following holds:
	\begin{enumerate}
		\item $G_0$ is $M$-generic for $\pp$ and $G_1$ is $M[G_0]$-generic for $\pp$.
		\item $G_1$ is $M$-generic for $\pp$ and $G_0$ is $M[G_1]$-generic for $\pp$.
	\end{enumerate}
\end{thm}

\begin{cor}\label{prodForce}
	Forcing with $\pp_2 \times \pp_2$ is equivalent to forcing with the product~$\pp \times \pp \times \pp \times \pp$, which is equivalent to iterating classical Cohen forcing four times over $M$.
\end{cor}

After forcing with $\pp_2 \times \pp_2$, we denote the resultant $M$-generic tuple by $(\alpha_0,\alpha_1,\alpha_2,\alpha_3)$, where \cref{prodForce} implies that $\alpha_0$ is $M$-generic, $\alpha_1$ is $M[\alpha_0]$-generic, $\alpha_2$ is $M[\alpha_0,\alpha_1]$-generic, and so forth.

\begin{lem}\label{lem:afterMe}
	There exist $p = I \times J$ and $q = I' \times J'$ and $g \in G$ such that $\diam(q) > \diam(p)$, $I$ lies in the first quadrant, and if $(\alpha_0,\alpha_1,\alpha_2,\alpha_3) \in p \times q$ is $M$-generic then $C((\alpha_0,\alpha_1),(\alpha_2,\alpha_3)) = g$. 
\end{lem}

In other words, some condition forces $C$ to be constant on all generics inside it. For notation, we say $D(I \times J) < D(I' \times J')$ if $I \subsetneq I'$ and $J \subsetneq J'$.

\begin{proof}
	Firstly, note that being a Borel code is $\Pi_1^1$ \cite[II.25.44]{MR1940513}, and hence absolute by Shoenfield absoluteness \cite[II.25.20]{MR1940513}. Therefore, the interpretation $C^M$ of $z$ in $M$ is a total function from $\rr^M$ to $G$ (recalling that $G = \set{g_i}{i < \omega} \in M$), and this fact is forced by the empty condition. Hence we may assume $p$ extends some condition in the first quadrant.
	
	Let $(\dot{\alpha}_0,\dot{\alpha}_1,\dot{\alpha}_2,\dot{\alpha}_3)$ be a name for the $M$-generic. By the fundamental theorem of forcing, there exists a condition $p \Vdash C(\dot{\alpha}_0,\dot{\alpha}_1,\dot{\alpha}_2,\dot{\alpha}_3) = \check{g}_i$ for some $i < \omega$. Write $p = I \times J \times I' \times J'$. Of course, we can assume $D(I \times J)$ to be smaller than $D(I' \times J')$. By \cref{prodForce}, this suffices.
\end{proof}

From now on, with $M$ fixed, we introduce the following notation.

\begin{dfn}
If $p \in \pp$ let $p^{*}$ denote the set (in $V$) of $M$-generics contained in the condition~$p$. We also relativise: if $x$ is a real then $p^{*[x]}$ denotes the set of $M[x]$-generics contained in $p$.
\end{dfn}

\begin{rem}
	A real is $M$-generic if and only if it intersects every dense open subset of $\pp$ that is an element of $M$. Since every open set in $\pp$ (and hence in $\pp_2$, etc.) has a Borel code (since $\rr^n$ is Polish), the following observation is immediate: since $\rr$ is Baire, the class of $M$-generics for $\pp$ is a comeagre subset of $\rr$ \emph{as viewed in $V$}. Hence, every condition $p$ contains a comeagre-in-$p$ set of $M$-generics.
\end{rem}

\section{The Case $N = 2$}\label{sec:thm1}

In this section, we prove that every Borel definable group extension of~$(\rr^2,+)$ by a group $G$ can be trivialised by a Borel coboundary. Let
	\[
		C \colon (\rr^2)^2 \rightarrow G
	\]
be a Borel cocycle. We argue in a sequence of lemmas which eventually reduces $C$ to a sum of Borel coboundaries.

\medskip

Initially, we follow closely the structure laid out in \cite{MR1841758}, and the lemmas of algebraic type carry over from their work. As an example, the base case in \cref{lem:sum} is virtually identical to their arguments. However, the lemmas needed for our $2$-dimensional (and beyond) proof, viz.\ \cref{usefulLemma} and \cref{lem:allCases}, are more technical due to the topological and geometrical differences between $\rr$ and $\rr^2$. This is particularly emphasised in \cref{sec:final}, where we prove the general case $N > 2$.

\subsection{Geometrical complications} In the process of trivialising the Borel cocycle $C$, it is our goal to show that $C$ is well-behaved on large and well-behaved (e.g.\ closed) subsets of $\rr^2$. We usually argue via generics (as there is a large---in fact comeagre---set of them in $V$) and show that $C$ is well-behaved on generics first; then we reduce to non-generics.

The following \cref{lem:sum}---an adaptation of Lemma 52 in \cite{MR1841758}---is the first step in this reduction process. We break it down into two cases:

\begin{lem}\label{lem:sum}
	Suppose $n \in \{ 2,3 \}$, and assume that $x_1,y_1,\ldots,x_n,y_n \in p^*$ satisfy $\sum_{i\leq n} x_i = \sum_{i \leq n} y_i$. Then $C(x_1,\ldots,x_n) = C(y_1,\ldots,y_n)$.
\end{lem}

The proof is very similar to that of \cite{MR1841758}; however, in our context, it requires the following geometrical lemma (which is trivial on $\rr$ and hence was not required in~\cite{MR1841758}).

\begin{lem}\label{usefulLemma}
	For $x_1,y_1,x_2,y_2,x_3,y_3 \in p^*$ satisfying $x_1 + x_2 + x_3 = y_1 + y_2 + y_3$, there exists a non-zero~$\epsilon \in \rr^2$ for which
\[
	\{ x_1 + \epsilon, x_2 - \epsilon, y_2 + \delta \} \subset p
\]
where $\delta = (y_1 - x_1) - \epsilon$. This $\epsilon$ can be chosen to be $\{ x_1,y_1,x_2,y_2,x_3,y_3 \}$-generic over $M$.
\end{lem}

Let $\pi_0$ denote the projection function onto the first coordinate in $\rr^2$, and define $\pi_1$ similarly.

\begin{proof}
	Suppose we are given $x_1,y_1,x_2,y_2,x_3,y_3 \in p^*$. First, assume that
	\begin{align*}
		\pi_0(x_1) = \pi_0(y_1) \text{\; and \;} \pi_1(x_1) > \pi_1(y_1)
	\end{align*}
where the choice of indices and coordinates is made without loss of generality. There are two cases.
\begin{itemize}
\item[(i)] First, assume the two remaining pairs can be split up into $(x_2,y_3)$ and $(x_3,y_2)$ both not pointing downwards, and let $D_1 = \pi_1(x_1) - \pi_1(y_1)$. Since $x_1 + x_2 + x_3 = y_1 + y_2 + y_3$, we may assume w.l.o.g.\ that $\pi_1(y_2) - \pi_1(x_3) \geq D_1/2$. Let $\epsilon = (0,\pi_1(x_2) - \pi_1(y_3) - \gamma)$ where $\gamma > 0$ is sufficiently small and generic (such a $\gamma$ exists since $p$ is open). It is easily verified that this $\epsilon$ works.
\item[(ii)] Second, suppose only one of the two remaining pairs points upwards, w.l.o.g.\ $(x_2,y_2)$; hence $\pi_1(y_2) - \pi_1(x_2) \geq D_1$. Then choose $\epsilon = (0,-\gamma)$ where $0 < \gamma < D_1/2$ is sufficiently generic. Again, it is easily verified that this $\epsilon$ works.
\end{itemize}
Now suppose that no pair $(x_i,y_j)$ out of $x_1,y_1,x_2,y_2,x_3,y_3$ can be arranged to share a coordinate. We argue geometrically: since $\sum_{i\leq 3} x_i = \sum_{i\leq 3} y_i$, we know that $\sum_{i\leq 3} (y_i - x_i) = 0$; hence consider the triangle $\Delta$ whose sides are given by the vectors $(x_i - y_i)$. Its set of vertices $V$ is identified with the pairs of vectors that induce the edges. In particular, let $v$ be the vertex of $\Delta$ where $(y_2 - x_2)$ and $(y_1 - x_1)$ meet.

Let $R_i$ be the unique rectangle with diagonal $(y_i - x_i)$ and sides parallel to the horizontal and vertical axes; then $\bigcup_{i\leq 3} R_i$ bounds a rectangle inside~$\Delta$. Further, each $R_i$ is a shift of the unique condition~$p_i$ bounding the pair $(x_i,y_i)$.

By an easy geometrical observation, and w.l.o.g., we have $R_1 \cap R_2 \neq \emptyset$, and $R_2 \cap R_3 = \emptyset$. Now, let $\epsilon = \beta + \gamma$ where $\beta \in R_1 \cap R_2$, the shift $\gamma$ is sufficiently generic, and the following hold:
\begin{enumerate}
	\item $x_1 + \epsilon \in p$
	\item $x_2 - \epsilon \in p$
	\item $v + \epsilon \in R_3$
\end{enumerate}
	A sufficiently small $\gamma$ exists since $p$ is open. By symmetry of $R_3$ (and hence of $p_3$), since $v + \epsilon \in R_3$, we have $y_3 + \delta \in p_3 \subset p$. Relabelling $y_3$ to $y_2$ completes the proof.\qedhere
\end{proof}

\begin{proof}[Proof of \cref{lem:sum}]
	For $n=2$, the proof is identical to the argument in~\cite{MR1841758}, which we include here for completeness. Let $x_1,x_2,y_1,y_2 \in p^*$ satisfy $x_1 + x_2 = y_1 + y_2$. Take a sufficiently small condition $q < p$ such that there exists a generic $z \in q^{*[x_1,x_2,y_1,y_2]}$ for which $z' = z + (y_1 - x_1) \in q$. This exists since $p$ is open. Now $C(x_1,x_2,z,z') = C(y_1,y_2,z,z')$ by the following argument: by the properties of abelian cocycles,
	\begin{align*}
		C(x_1,x_2,z,z') &= C(x_1,z') + C(x_2,z) + C(x_1 + z', x_2 + z)\\
		&= 2g + C(x_1 + z', x_2 + z)
		\intertext{since the pairs in the first two terms are generic (cf.\ \cref{prodForce}). Similarly,}
		C(y_1,y_2,z,z') &= C(y_1,z) + C(y_2,z') + C(y_1 + z, y_2 + z')\\
		&= 2g + C(y_1 + z, y_2 + z').
	\end{align*}
	Now recall that $x_1 + x_2 = y_1 + y_2$; hence $x_2 + z = y_2 + z'$, and the equality $z' + x_1 = z + y_1$ is immediate from the definition. Further,
\begin{align*}
	C(x_1,x_2,z,z') &= C(x_1,x_2) + C(z,z') + C(x_1+x_2,z+z')\\
	&= C(y_1,y_2) + C(z,z') + C(y_1+y_2,z+z')\\
	&= C(y_1,y_2,z,z')
\end{align*}
from which the case $n=2$ follows immediately.

For $n=3$, we use \cref{usefulLemma}: let $(\epsilon,\delta)$ be as provided by the lemma. Then define, as in \cite{MR1841758}, the sequences $(x'_i)$ and $(y'_i)$ such that $x'_3 = y'_3$ and
\[
	x'_1 = x_1 + \epsilon, \; x'_2 = x_2 - \epsilon, \; y'_1 = y_1 - \delta, \; y'_2 = y_2 + \delta
\]
which yields the result at once from the special case above and the properties of abelian cocycles.
\end{proof}

We prove below the full version of \cref{lem:sum}. Due to the topological differences between~$\rr$ and~$\rr^2$, the general case of our proposition below is more involved than Kanovei and Reeken's counterpart.

\begin{prp}\label{lem:allCases}
	\Cref{lem:sum} holds in fact for all $n \geq 2$.
\end{prp}

We will prove the proposition by induction; consider case the $k+1$.  Before we give the intricate general proof, we reproduce the \emph{special case} from \cite{MR1841758}, which also applies in the case $N=2$ (and in fact for every $N$): we assume that $x_\ell = y_\ell$ for some $\ell \leq k+1$. To that end, suppose the proposition holds below $k+1$, let $x_1,y_1,\ldots,x_{k+1},y_{k+1} \in p^*$ be given, and suppose w.l.o.g.\ that $x_1 = y_1$. Then the lemma follows immediately:
\begin{align*}
	C(x_1,\ldots,x_{k+1}) &= C(x_2,\ldots,x_k) + C(x_2 + \ldots + x_{k+1}, x_1)\\
	&= C(y_2,\ldots,y_k) + C(y_2 + \ldots + y_{k+1}, y_1) = C(y_1,\ldots,y_{k+1})
\end{align*}
by the inductive hypothesis and the definition of $C$ for more than two inputs, as required.

\medskip

The general proof of \cref{lem:allCases} is explained geometrically by considering the $k+1$-gon~$\Gamma$ constructed by concatenating the vectors $(y_1 - x_1),\ldots, (y_{k+1} - x_{k+1})$ (recall that the specific labelling of vectors is inconsequential). As before, for every $(y_i - x_i)$, let $R_i$ denote the rectangle bounding it, which can be considered the translation of an associated condition $p_i \subset p$. Since $p$ is open, we may extend $p_i$ by a small amount in all directions without leaving $p$; we use this fact repeatedly.

First, we prove another geometrical lemma.

\begin{lem}\label{lem:ngonLem}
	With $x_1,y_1,\ldots,x_{k},y_k \in p^*$ and $k \geq 3$, its associated $k$-gon~$\Gamma$ can be transformed into a $k$-gon~$\Gamma'$ containing an edge parallel to the second axis without changing the value of $C$ on $\Gamma'$.
\end{lem}

\begin{proof}
	Let $\Gamma$ be the $k$-gon induced by the vectors $(y_i - x_i)$. W.l.o.g.\ let $(y_1 - x_1)$ be the vector with the smallest positive $\pi_0$-shift; that is, $0 < b = \pi_0(y_1)-\pi_0(x_1)$ is minimal among all pairs $(x_i,y_i)$. Since $\sum x_i = \sum y_i$, there exists $(y_2 - x_2)$ for which $\pi_0(x_2) > \pi_0(y_2)$. Write $b = e + d$ where $e$ is sufficiently generic, and define
	\begin{align*}
		x_1' = x_1 + (e,0) \; &,  \; x_2' = x_2 - (e,0)\\
		y_1' = y_1 - (d,0)  \; &, \; y_2' = y_2 + (d,0)
	\end{align*}
	and put $x_i' = x_i$ and $y_i' = y_i$ everywhere else. Now, all $x'_i,y'_i \in p^*$ by minimality of $b$, and further, $\sum x_i' = \sum x_i$ and $\sum y_i' = \sum y_i$. But since $k \geq 3$, there exists $j$ for which $x_j = x_j'$ and $y_j = y_j'$. By the special case,
	\[
		C(x_1,\ldots,x_{k}) = C(x_1',\ldots,x_k') \; \text{ and } \; C(y_1,\ldots,y_{k}) = C(y_1',\ldots,y_k').
	\]
	Finally, the vector $(y_1' - x_1')$ is vertical since $\pi_0(x_1') = \pi_0(x_1) + e = \pi_0(y_1) - d = \pi_0(y_1')$.
\end{proof}

By symmetry, the same argument holds for the first axis (or any axis in $\rr^N$). The following proposition can be considered a repeated application of \cref{lem:ngonLem}, which ``smooths'' the $k$-gon $\Gamma$ induced by the set of points $x_1,y_1,\ldots,x_{k+1},y_{k+1}$ until one of its sides vanishes.

\begin{proof}[Proof of \cref{lem:allCases}]
Suppose $x_1,y_1,\ldots,x_{k+1},y_{k+1} \in p^*$ and there is no pair $(i,j)$ for which $x_i = y_j$; hence the special case does not apply. By \cref{lem:ngonLem}, we may assume the vector $(y_1-x_1)$ to be vertical, and assume w.l.o.g.\ that $\pi_1(x_1) > \pi_1(y_1)$.
By minimality, there exists $n < k+1$ and $i_1,\ldots,i_n$ such that
\[
	\sum_{j \leq n} (\pi_1(y_{i_j}) - \pi_1(x_{i_j})) \geq \pi_1(x_1) - \pi_1(y_1)
\]
and $\pi_1(y_{i_j}) - \pi_1(x_{i_j}) > 0$ for all $j \leq n$. We now build sequences $(y_i^{(\ell)} - x_i^{(\ell)})$ and $(e^{(\ell)},d^{(\ell)})$ for $1 \leq \ell \leq n$ as follows. First, let $x_i^{(1)} = x_i$ and $y_i^{(1)} = y_i$ for all $i \leq k+1$. We construct the sequence by recursion on $\ell$: with $(y_i^{(\ell)} - x_i^{(\ell)})$ given, consider the following two cases:
\begin{enumerate}
	\item[(a)] Suppose there exist $e,d \in (0,\pi_1(y_{i_\ell}) - \pi_1(x_{i_\ell}) + a)$ (where $a$ is sufficiently small such that $p_{i_\ell} + (0,a) \subset p$) for which $e$ is sufficiently generic and
	\[
		x_1^{(\ell)} - (0,e) = y_1^{(\ell)} + (0,d).
	\]
	Then let $(e^{(\ell)},d^{(\ell)}) = (e,d)$ and stop the recursion.
	\item[(b)] If no such pair exists, pick $e \in (\pi_1(y_{i_\ell}) - \pi_1(x_{i_\ell}),\pi_1(y_{i_\ell}) - \pi_1(x_{i_\ell}) + a)$, sufficiently generic, and put $(e^{(\ell)},d^{(\ell)}) = (e,e)$.
\end{enumerate}
Now define:
\begin{align*}
		x_1^{(\ell+1)} = x_1^{(\ell)} - (0,e^{(\ell)}) \; &,  \; x_{i_\ell}^{(\ell+1)} = x_{i_\ell}^{(\ell)} + (0,e^{(\ell)})\\
		y_1^{(\ell+1)} = y_1^{(\ell)} + (0,d^{(\ell)}) \; &,  \; y_{i_\ell}^{(\ell+1)} = y_{i_\ell}^{(\ell)} - (0,d^{(\ell)})
	\end{align*}
The following claims ensure that the result follows.

\begin{claim}
	$C(x_1,\ldots,x_{k+1}) = C(x_1^{(\ell)},\ldots,x_{k+1}^{(\ell)})$ for all $\ell \leq n$.
\end{claim}

\begin{cproof}
	The construction shows that $\sum_i x_i = \sum_i x_i^{(\ell)}$. Since at least one value $x_j^{(\ell)}$ is preserved when passing from $x_j^{(\ell)}$ to $x_j^{(\ell+1)}$, induction proves the result from the special case. The same argument holds for the sequences $(y_i^{(\ell)})$. Since $\sum x_i = \sum y_i$, we have for all $\ell,\ell'$
	\[
		\sum_i x_i^{(\ell)} = \sum_i y_i^{(\ell')}. \qedhere
	\]
\end{cproof}

\begin{claim}
	Case (a) above occurs in fewer than $n$ steps.
\end{claim}

\begin{cproof}
	Let $D^{(1)} = \pi_1(x_1) - \pi_1(y_1)$. Suppose Case (a) does not occur. Then, by construction,
	\[
		e^{(1)} > \pi_1(y_{i_1}) - \pi_1(x_{i_1})
	\]
	and so
	\[
		D^{(2)} = \pi_1(x_1') - \pi_1(y_1') < D^{(1)} - e^{(1)}.
	\]
	Since $\sum \pi_1(y_{i_j}) - \pi_1(x_{i_j}) \geq \pi_1(x_1) - \pi_1(y_1)$, the result follows by induction.
\end{cproof}

\begin{claim}
	When Case (a) is applied at stage $\ell +1$, then $C(x_1^{(\ell+1)},\ldots,x_{k+1}^{(\ell+1)}) = C(y_1^{(\ell+1)},\ldots,y_{k+1}^{(\ell+1)})$.
\end{claim}

\begin{cproof}
	By construction,
	\[
		x_1^{(\ell+1)} = x_1^{(\ell)} - (0,e^{(\ell)}) = y_1^{(\ell)} + (0,d^{(\ell)}) = y_1^{(\ell+1)}
	\]
	and by Claim 1 and the special case the result follows.
\end{cproof}

Combining the claims yields the sequence of equalities
\[
	C(x_1,\ldots,x_{k+1}) = C(x_1^{(\ell+1)},\ldots,x_{k+1}^{(\ell+1)}) = C(y_1^{(\ell+1)},\ldots,y_{k+1}^{(\ell+1)}) = C(y_1,\ldots,y_{k+1})
\]
which completes the proof. \qedhere
\end{proof}

\subsection{Trivialising in stages}\label{sec:trivStages}
The previous proposition shows that $C$ is well-behaved under equal sums of sequences of equal length. In order to fully trivialise $C$ to a Borel coboundary, we require more, both algebraically and combinatorially. Most of the arguments in this subsection are of algebraic type, and hence carry over from~$\rr$ to $\rr^2$; in this subsection, we reproduce lemmas from \cite{MR1841758} with additional detail, and follow their general proof structure.

Accordingly, we introduce some notation: given a sequence $x_1,\ldots,x_n$, denote it by $(\bar{x})_n$, and hence let $\sum (\bar{x})_n = x_1 + \ldots + x_n$. We also write $(\bar{x})_n \in p^*$ if $x_i \in p^*$ for all $i \leq n$, and $C((\bar{x})_n) = C(x_1,\ldots,x_n)$. When the length of $(\bar{x})_n$ is clear, we also simply write $\bar{x}$ in its place. Further, if $k < \omega$ and $g \in G$, write $kg = g + \ldots + g$ containing $k$-many terms.

\medskip

The following two combinatorial lemmas are due to \cite{MR1841758}. As results in this subsection, as can be seen from their proofs, hold for all $\rr^N$; we state some without proof and refer to \cite{MR1841758} for details.

\begin{lem}[{\cite{MR1841758}, Lemma 53}]
	Suppose $(\bar{x})_n,(\bar{y})_m,(\bar{x}')_{n'},(\bar{y}')_{m'} \in p^*$ where $1 \leq m < n$ and $1 \leq m' < n'$. Suppose $\sum (\bar{x})_n = \sum (\bar{y})_m = s$ and $\sum (\bar{x}')_{n'} = \sum (\bar{y}')_{m'} = s'$. Then
	\[
		(n' - m')(C(\bar{x}) - C(\bar{y})) = (n - m)(C(\bar{x}') - C(\bar{y}')).
	\]
\end{lem}

\begin{lem}\label{lem45}
	There exist $\tilde{x},\tilde{y} \in p^*$ and $n < \omega$ such that $(n+1)\tilde{x} = n\tilde{y}$.
\end{lem}

\begin{proof}
	Fix any $x \in p^*$ and consider the straight line $L$ containing the origin and $x$. Since $p$ is open, there exists $n < \omega$ for which $\left(1 + \frac{1}{n}\right)x \in p^* \cap L$, which yields the desired $y$.
\end{proof}

\begin{dfn}
	The pair of sequences $(\bar{x})_{(n+1)}$ and $(\bar{y})_{n}$ given by $x_i = \tilde{x}$ and $y_i = \tilde{y}$ for all $i$ is the \define{canonical pair}, denoted by $(\lambda,\mu)$. The \define{canonical group element} is given by
	\[
		\tilde{g} = C(\lambda) - C(\mu).
	\]
\end{dfn}

\begin{lem}\label{lem765}
	For any $(\bar{x})_n, (\bar{y})_m \in p^*$ with $n \geq m$ and $\sum \bar{x} = \sum \bar{y}$, we have $C(\bar{x}) - C(\bar{y}) = (n-m)\tilde{g}$.
\end{lem}

\begin{proof}
	\Cref{lem45} implies $(n+1-n)(C(\bar{x}) - C(\bar{y})) = (n-m)(C(\lambda) - C(\mu)) = (n-m)\tilde{g}$.
\end{proof}

Following \cite{MR1841758}, we define $C_1 \colon (\rr^2)^2 \rightarrow G$ given by
\[
	C_1(x,y) = C(x,y) - \tilde{g}.
\]
By induction, we have $C_1((\bar{x})_{(n+1)}) = C((\bar{x})_{(n+1)}) - n\tilde{g}$. Now, observe that $C_1(x,y) - C(x,y) = -\tilde{g}$, a Borel coboundary: to make this explicit, write $B(x) = -\tilde{g}$ so that
\[
	C_B(x,y) = B(x) + B(y) - B(x+y) = -\tilde{g}	.
\]
Hence it suffices to prove that $C_1$ is a Borel coboundary. Luckily, $C_1$ satisfies stronger closure properties than $C$.

\begin{lem}[{\cite{MR1841758}} Corollary 54]\label{lem:KR2}
	If $(\bar{x})_n, (\bar{y})_m \in p^*$ satisfy $\sum (\bar{x})_n = \sum (\bar{y})_m$ and $n \geq m$ then $C_1(\bar{x}) = C_1(\bar{y})$.
\end{lem}

\begin{proof}
	By the inductive remark, \cref{lem765} proves
	\begin{align*}
		C_1(\bar{x}) - C_1(\bar{y}) &= C(\bar{x}) - C(\bar{y}) - (n-1)\tilde{g} + (m-1)\tilde{g}\\
		&= C(\tilde{x}) - C(\tilde{y}) - (n-m)\tilde{g} = 0. \qedhere
	\end{align*}
\end{proof}

\subsection{Trivialising on a closed set}\label{sec:trivClosedSet}
From the difference in complexity between our \cref{lem:allCases} and Lemma 52 in \cite{MR1841758}, it is clear that the case $\rr^2$ is not an immediate consequence of the theorem in the case $\rr$. The topological differences between the spaces are emphasised in this section, where we overcome obstacles not present in the 1-dimensional case. These can be reduced to the properties of a set $t(x)$ associated with each $x \in \rr^2$. Trivialising on $t(x)$---similar to trivialising on $[M,\infty)$ in \cite{MR1841758}---suffices to prove the theorem.

\medskip

We begin by identifying a well-behaved subset of $\rr^2$ on which the invariance properties of $C_2$ apply more broadly. For $k < \omega$, let
\[
	k \cdot p = \set{(kx_1,kx_2)}{(x_1,x_2) \in p}.
\]
	As we assumed that $p$ lies in the first quadrant, there exist $M < \omega$ and distinct straight lines $L_1,L_2$ such that:
\begin{itemize}
	\item there exists a straight line $L$ through the origin for which, for some $\tilde{z} = (z_1,z_2) \in \rr^2$, the infinite line segment
	\[
		L^{\tilde{z}} = \set{(x_1,x_2) \in L}{x_1 \geq z_1 \land x_2 \geq z_2}
	\]
	is contained in $p^+ = \bigcup_{k > M} k \cdot p$;
	\item $L_1$ and $L_2$ are neither equal nor parallel, $L_1 \cap L_2 \subset p^+$, and the region bounded by $L_1,L_2$  which is wholly contained in the first quadrant is also contained in $p^+$.
\end{itemize}

We show below that such points and lines always exist. We need some notation. In $\rr^N$, denote the origin by $\origin$. Given a condition $p$, denote its topological closure by $\overline{p}$, and its topological boundary by $\del{p}$. Let $d$ denote the Euclidean metric.

\begin{lem}\label{lem:allExist}
	Such $M < \omega$, straight lines $L,L_1,L_2$, and $\tilde{z} \in \rr^2$ exist.
\end{lem}

\begin{proof}
Consider the sequence $(k \cdot p)_{k < \omega}$ of conditions (read open sets), where each $k \cdot p$ has vertices~$v_k^1,v_k^2,v_k^3,v_k^4$, labelled anti-clockwise starting with the vertex closest to the origin. Since $p$ is open, we may assume that $p$ is not a square. Hence suppose that $\pi_0(v_1^2) - \pi_0(v_1^1) > \pi_1(v_1^3) - \pi_1(v_1^2)$ (i.e.\ ~$p$ is wider than it is high); the other case is similar. Define auxiliary straight lines~$L'_1,L'_2$ by
\[
	\{ \origin, v_1^1 \} \subset L'_1 \; \text{ and } \; 
	\{ \origin, v_1^3 \} \subset L'_2.
\]
Note that $\set{v^1_k}{k < \omega} \subset L'_1$ and $\set{v^3_k}{k < \omega} \subset L'_2$. Let $v'$ be the unique point at which $\del{p} \setminus \{ v_1^3 \}$ and $L'_2$ intersect. Now, define $w$ to be the unique $v \in \del{p}$ for which $d(v_1^1,v) = d(v',v)$ and for which this distance is minimal. That is, $w \in \del{p}$ and sits exactly halfway between $v_1^1$ and $v'$. Let $L$ be the straight line satisfying
\[
	\{ \origin, w \} \subset L.
\]
To see that for some $M < \omega$ we have $(M \cdot p) \cap ((M+1) \cdot p) \neq \emptyset$, assume w.l.o.g.\ that $p$ is wider than it is high. Write $v_1^1 = (x_1,x_2)$ and $v_1^4 = (y_1,y_2)$. Then $M < \omega$ is as needed if
\[
	Mx_1 < (M+1)x_1 < My_1 \iff M > \frac{x_1}{y_1-x_1}.
\]
To complete the proof, fix $\tilde{z} \in L \cap (M \cdot p)$ for a sufficiently large~$M$. Choose~$L_1$ to be the unique line parallel to $L'_1$ containing $\tilde{z}$, and define $L_2$ similarly.
\end{proof}

In the 2-dimensional case, the definition of $T$ is best given in polar coordinates: let $L_1$ be the unique straight line through $\origin$ with angle $\theta_1$, and define $L_2$ and $\theta_2$ similarly. Suppose that~$\tilde{z} = (s,\varphi)$.

\begin{dfn}\label{dfn:T}
	With $\tilde{z}$, $L_1$, $L_2$, and $L$ given as above, define
	\[
		T = \set{(r,\theta) \in \rr^2}{r \geq s \land \theta_2 \leq \theta \leq \theta_1}.
	\]
\end{dfn}

One can think of $T$ as the \emph{cone} induced by the lines $L_1',L_2'$ shifted to $\tilde{z}$. Observe that $\tilde{z} \in T$. By construction, we may choose $\tilde{z}$ to not be contained in $p$, which we assume from now on. Hence~$T \cap p = \emptyset$.

\begin{lem}\label{Tprops}
	With $T$ defined as above, the following hold:
	\begin{enumerate}[label=(\alph*)]
		\item\label{TWellDef} $T$ is non-empty and closed (hence in particular Borel).
		\item\label{TLz} $L^{\tilde{z}} \subset T$.
		\item\label{TAddition} $T$ is closed under addition.
		\item\label{TInfDiam} The interior of $T$ has infinite diameter: for every $K < \omega$ there exists $y \in T$ such that the open ball $B_K(y)$ of radius $K$ and centre $y$ is contained in $T$.
	\end{enumerate}
\end{lem}

\begin{proof}
	These are all routine. It is clear from the construction that (a) $T$ is non-empty and closed (it is the closure of the area bounded by $L_1$ and $L_2$ in the direction of $(\infty,\infty)$), and similarly~(b) is immediate. The fact (c) that~$T$ is closed under addition follows easily from the construction of $T$ and by using polar coordinates. Since $L_1$ and $L_2$ are not parallel, the interior of~$T$ must have infinite diameter, proving (d).
\end{proof}

Importantly, every real in $T$ can be expressed as a sum of generics:

\begin{lem}\label{sumGen}
	If $z \in T$ then there exists a sequence of generics $(\bar{x})_n$ for which $z = \sum \bar{x}$.
\end{lem}

\begin{proof}
	Let $z = 2x$ where $x \in p \setminus p^*$. We find $x_1,x_2 \in p^*$ for which $z = x_1 + x_2$. Let $r < p$ contain~$x$, and let $\epsilon$ be $M$-generic and sufficiently small so that $x_1 = x + \epsilon$ and $x_2  = x - \epsilon$ are elements of~$r^*$. Now $z = x_1 + x_2$ as required, and it follows by induction that every $z \in k \cdot p$ (for any $k \geq 2$) can be written in the required form. Since $T \subset p^+$, the lemma is proven.
\end{proof}

We begin to trivialise cocycles to coboundaries. As in \cite{MR1841758}, \cref{lem:KR2} shows that the map~$F \colon \rr^2 \rightarrow G$ defined by
\[
	F(x) = \begin{cases}
		C_1(x_1,\ldots,x_n) & \text{ if $x = \sum (\bar{x})_n \in T$}\\
		0 & \text{ otherwise}
	\end{cases}
\]
for $(\bar{x})_n \in p^*$ is well-defined. Further, as $C_1$ is Borel, the graph of $F$ is clearly analytic; hence $F$ is a Borel function and $C_F$ is a Borel coboundary. Now define
\[
	C_2(x_1,x_2) = C_1(x_1,x_2) + C_F(x_1,x_2).
\]

\begin{lem}\label{prevLem2}
	If $(\bar{x})_n \in p^*$ and $\sum (\bar{x})_n \in T$ then $C_2(\bar{x}) = 0$. 
\end{lem}

\begin{proof}
	By induction. For $n=2$, it is easily seen that $C_2(x_1,x_2) = F(x_1) + F(x_2) = 0$ by definition of $F$ and since $T \cap p = \emptyset$. Thus, \cref{alg1} implies
	\begin{align*}
	C_2(x_1,\ldots,x_{n+1}) &= C_1(x_1,\ldots,x_{n+1}) + C_F(x_1,\ldots,x_{n+1})\\
	&= C_1(x_1,\ldots,x_{n+1}) - F(x_1 + \ldots + x_{n+1}). \qedhere
	\end{align*}
\end{proof}

Importantly, the invariance of $C_2$ on generics in fact extends to all of~$T$, which follows from the combinatorial properties of abelian cocycles.

\begin{cor}[{\cite{MR1841758}} Lemma 55]\label{corRRRR}
	If $x,y \in T$ then $C_2(x,y) = 0$.
\end{cor}

\begin{proof}
	Suppose $x = \sum (\bar{x})_n$ and $y = \sum (\bar{y})_m$, all generics in $p$ (this is possible by \cref{sumGen}). By the algebraic properties of $C_2$ and the previous lemma we have
	\begin{align*}
		C_2(\bar{x},\bar{y}) - C_2(\bar{x}) - C_2(\bar{y}) &= C_2(x,y)
	\end{align*}
	where all terms on the left-hand side equal $0$.
\end{proof}

We are now ready to define the aforementioned class of families $t(x)$.

\begin{dfn}\label{dfntxxx}
For every $x \in \rr^2$, let $t(x) \subset \rr^2$ be defined by
\[
	t(x) = \set{y \in T}{x + y \in T}.
\]
\end{dfn}

\begin{lem}\label{lem:sump}
	Let $x,x',w \in \rr^2$.
	\begin{enumerate}[label=(\alph*)]
		\item $t(x) \neq \emptyset$
		\item $t(x) \cap t(x') \neq \emptyset$
		\item\label{t3} If $x,x' \in T$ then there exist $z,z' \in t(x) \cap t(x') \cap t(w)$ for which we have $x + z = x' + z'$.
	\end{enumerate}
\end{lem}

\begin{proof}
	For (a), let $x \in \rr^2$ be given. If $x \in T$ then $x+x \in T$, thus $x \in t(x)$. Similarly, if $x \in L \setminus T$ (with $L$  defined as in \cref{lem:allExist}), take a sufficiently large $y \in L^{\tilde{z}} \subset T$ satisfying $x + y \in T$. Hence suppose $x \in \rr^2 \setminus (T \cup L)$. Recall that $T$ has infinite diameter (\cref{Tprops}~\ref{TInfDiam}). Hence let $z \in T$ be such that $B_{|x| + 1}(z) \subset T$. Then $x + z \in B_{|x| + 1}(z) \subset T$. Part (b) follows easily; in fact, the result holds for all finite intersections by a similar argument.
	
	For part (c), let $T_x = \set{x + y}{y \in T}$ be the shift of $T$ to $x$. Since $L_1$ and $L_2$ (as defined in \cref{lem:allExist}) are not parallel, it is easily seen that if $x,x' \in T$ then $T_x \cap T_{x'} \neq \emptyset$, and in fact also has infinite diameter. As in (a) above, let $v \in T$ be such that $B_{|w| + 1}(v) \subset T_x \cap T_{x'}$. Then any element in~$B_{|w| + 1}(v)$ can be decomposed as desired.
\end{proof}

Since the set $T$ is closed by construction, the following is immediate.

\begin{cor}\label{lem:txCompl}
	For every $x \in \rr^2$, the set $t(x)$ is Borel.
\end{cor}

Define $J(x) = C_2(x,\tilde{x})$ where $\tilde{x} \in t(x)$. By our construction, the choice of $\tilde{x}$ is inconsequential, highlighting the $C_2$-invariance on $T$. To see this, first recall that
\[
	C_2(x,y) = C_1(x,y) - F(x+y) = C(x,y) - \tilde{g} - F(x+y) = C_2(y,x)
\]
since $\rr^2$ and $G$ are abelian, and since $C$ is an abelian cocycle. 

\begin{lem}\label{functionLDfn}
	$J$ is well-defined and a Borel function.
\end{lem}

\begin{proof}
	By \cref{algProp1},
	\[
		J(x) = C_2(x,\tilde{x}) = C_2(\tilde{x},x) = C_2(\tilde{x},z) + C_2(\tilde{x}+z,x) - C_2(\tilde{x} + x,z)
	\]
for every $z \in \rr^2$. Hence if $z \in t(x)$ then \cref{Tprops}~\ref{TAddition} and \cref{corRRRR} imply $J(x) = C_2(\tilde{x} + z,x)$. Now suppose $\tilde{x},\tilde{x}' \in t(x)$, and, using \cref{lem:sump}~\ref{t3}, let $z,z' \in t(x) \cap t(\tilde{x}) \cap t(\tilde{x}')$ be such that $\tilde{x} + z = \tilde{x}' + z'.$ Then
\[
	C_2(\tilde{x},x) = C_2(\tilde{x} + z,x) = C_2(\tilde{x}' + z',x) = C_2(\tilde{x}',x).
	\]

To see that $J$ is Borel as a map, first note that $C_2$ is Borel by construction. Hence it suffices to show that the set $\set{t(x)}{x \in X}$ has a Borel choice (or \emph{uniformisation}) function. To that end, let $P \subset (\rr^2)^2$ be the relation
\[
	(x,y) \in P \iff y \in t(x).
\]
Clearly, $P$ is a Borel subset of $(\rr^2)^2$ and, fortunately, every section $P_x$ has a canonical witness, which we identify as follows: fix $x$ and travel along $L$ starting at $\tilde{z}$ towards the point at infinity~$(\infty,\infty)$. Then the first $y \in L^{\tilde{z}}$ for which a sufficiently large closed ball centred at $y$ is contained in~$T$ does the trick (recalling that $T$ is closed shows that this is well-defined).

To formalise this, consider $L^{\tilde{z}}$ and define the function $\rho \colon \rr^2 \rightarrow \rr^2$ by
\begin{align*}
	\rho(x) = \min_{|y|}\Set{y \in L^{\tilde{z}}}{\overline{B}_{|x| + 1}(y) \subset T}
\end{align*}
where $\overline{B}_r(x)$ denotes the closed ball of radius $r$ and centre~$x$. It follows from the properties of~$T$ from \cref{Tprops} that $\rho(x)$ is well-defined and as desired. To see that $\rho$ is Borel, recall that a function is Borel if and only if the pre-image of every basic open set is Borel \cite[11.C]{MR1321597}. Fix a condition $q \subset \rr^2$; we show $\rho^{-1}[q]$ is Borel using \cref{prp:BorelDefn}. We may assume that $q \cap L^{\tilde{z}} \neq \emptyset$ since otherwise $\rho^{-1}[q] = \emptyset$. By definition of $\rho$ we know that $\tilde{z} \not\in \ran(\rho)$; hence~$q \cap L^{\tilde{z}}$ is homeomorphic to the open interval~$(0,1)$. Let
\[
	b = (\tilde{z} \oplus \tilde{q}) \oplus (\varphi \oplus \theta_1 \oplus \theta_2)
\]
where $\tilde{q}$ codes the limit points of $q \cap L^{\tilde{z}}$; we use this as an oracle to apply \cref{prp:BorelDefn}. Given $x \in \rr^2$, first compute $|x| + 1$. From the limit points $c_1,c_2 \in L^{\tilde{z}}$ of $q \cap L^{\tilde{z}}$ (which are coded into~$\tilde{q}$) compute the largest radii $r_1,r_2$ such that $\overline{B}_{r_1}(c_1) \subset T$ and $\overline{B}_{r_2}(c_2) \subset T$, which requires finitely-many jumps from $b$ (recall that $|c_1| < |c_2|$). By construction, $r_1 < r_2$, and if $r_1 < |x| + 1 < r_2$ then $x \in \rho^{-1}[q]$, which proves the result by \cref{prp:BorelDefn}.
\end{proof}

To complete the proof of \cref{thm:myThm1}, it suffices to show that $C_2$ is a Borel coboundary. For each $x \in \rr^2$ let
\[
	\tilde{x} = \rho(x) \in t(x)
	\]
which we call the \define{canonical witness for $x$}.

\begin{lem}\label{lem:last}
	Let $x,y \in \rr^2$. If $z \in t(x) \cap t(x+y)$ then $z + x \in t(y)$.
\end{lem}

\begin{proof}
	By definition, $z \in t(x)$ implies $z + x \in T$. Since $z \in t(x+y)$ we also know that $z + x + y \in T$. Hence $z + x \in t(y)$.
\end{proof}

Suppose $x,y \in \rr^2$ and choose $z \in t(x) \cap t(y) \cap t(x + y)$. Then
\begin{align*}
	C_2(x,y) &= C_2(x,z) + C_2(x+z,y) - C_2(x+y,z)\\
	&= C_2(x,\tilde{x}) + C_2(x+z,y) - C_2(x+y,\widetilde{x+y})\\
	&= J(x) + C_2(x+z,y) - J(x+y).
\end{align*}
Finally, notice that $z \in t(x) \cap t(x + y)$ implies $x + z \in t(y)$ by \cref{lem:last}. Hence
\[
	C_2(x+z,y) = C_2(y,x+z) = C_2(y,\tilde{y}) = J(y)
\]
and so $C_2(x,y) = C_J(x,y)$. Now,
\[
	C_J(x,y) = C_2(x,y) = C_1(x,y) + C_F(x,y) = C(x,y) + C_B(x,y) + C_F(x,y)
\]
and thus $C(x,y) = C_J(x,y) - C_B(x,y) - C_F(x,y)$, a sum of Borel coboundaries. This suffices by \cref{lastAlgPrp}. Hence we conclude:

\begin{thm}\label{thm:myThm1}
	For every group $G$ we have $\hbor\left(\rr^2,G\right) = 0$.
\end{thm}

\section{The Case $N > 2$}\label{sec:final}

Let $N > 2$ and $G$ be a group. Suppose $C \colon \left(\rr^N\right)^2 \rightarrow G$ is an abelian Borel cocycle with Borel code $z$ and let $M$ be a countable transitive model of $\ZFC$ containing $z$ and $G$ as elements. As in \cref{subsec}, we assume there exist $N$-dimensional conditions $p = I_1 \times \ldots \times I_N \subset \rr^N$ and $q = J_1 \times \ldots \times J_N$ satisfying that each $I_i \subset (0,\infty)$ (so $p$ lies in the $N$-dimensional ``first'' quadrant, octant, etc.) and so that there exists $\tilde{g} \in G$ for which $C(x) = \tilde{g}$ for every $x \in (p \times q)^*$. Further, we may of course assume $D(p) < D(q)$ (cf.\ after \cref{lem:afterMe}). We also stipulate that $|I_i| \neq |I_j|$ for all~$i \neq j$, which is needed in \cref{geomLem54366} below.

We retrace the arguments from \cref{sec:thm1}. In order to prove the theorem along our arguments in the previous section, we need to prove higher dimensional versions of our topological lemmas. We begin with a high-dimensional version of \cref{lem:allCases}:

\begin{prp}
	If $(\bar{x})_n, (\bar{y})_n \in p^*$ and $\sum \bar{x} = \sum \bar{y}$, then $C(\bar{x}) = C(\bar{y})$.
\end{prp}

\begin{proof}
	As before, the \emph{special case} is combinatorial and hence applies: if there exists $\ell \leq n$ for which $x_\ell = y_\ell$ then the lemma follows immediately. If $n \in \{ 2,3 \}$, then we are also done: if $n=2$, the combinatorial argument in \cite{MR1841758} applies; if $n=3$, then the vectors $(y_1 - x_1), (y_2 - x_2), (y_3 - x_3)$ bound a triangle~$\Delta$. Since $p$ is open, the same argument as in \cref{usefulLemma} holds, considering~$\Delta$ in its induced plane inside $p$. Hence, assume $n > 3$ and that there exists no pair $(i,j)$ for which $x_i = y_j$. Note that the arguments in \cref{lem:ngonLem} and the recursion in the proof of \cref{lem:allCases} do not depend on the number of dimensions in the ambient space; they can be carried out identically in $\rr^N$ for any~$N > 2$.
\end{proof}

Recall that the results in \cref{sec:trivStages} are purely combinatorial, hence apply immediately in the case $N > 2$. The remaining argument reduces to translating \cref{sec:trivClosedSet} to $N$. In fact, the only obstacle is to correctly define~$T$; all subsequent arguments are combinatorial.

\medskip

Recall the case $N=2$ and denote the set $T$ defined in \cref{sec:thm1} by $T(2)$, to highlight its topological dimension. We defined two auxiliary lines $L_1',L_2'$ which induce $T(2)$: it is the space bounded by $L_1',L_2'$ in the direction of the limit point $(\infty,\infty)$. Note that $T(2)$ is homeomorphic to the closed upper half plane, which we denote by
\begin{align*}
	U^2 &= \set{(x_1,x_2)}{x_2 \geq 0}.\\
\intertext{For each $N > 2$, we define a set $T(N)$ homeomorphic to~$U^N$, the~$N$-dimensional closed upper half plane given by}
	U^N &= \set{(x_1,\ldots,x_N)}{x_N \geq 0}.
\end{align*}
With the condition~$p \subset \rr^N$ fixed as before, it is easily seen that there exists~$M < \omega$ for which $(M \cdot p) \cap ((M + 1) \cdot p) \neq \emptyset$ (cf.\ \cref{lem:allExist}). Indeed, consider $p$ as an $N$-dimensional hypercube in $\rr^N$. By assumption, all $2^N$-many vertices of $p$ are positive. The existence of $M < \omega$ is derived exactly as in the two-dimensional case: the shortest side of $p$ gives a sufficient bound for $M$ just as in the proof of \cref{lem:allExist}. Hence, fix such an $M$, and consider the induced set~$p^+$.

From now on, we write $X \homeo Y$ to mean that $X$ and $Y$ are homeomorphic. The set of straight lines containing the origin $\origin$ is given by $\olines$.

\begin{prp}\label{geomLem54366}
	There exists a subspace $T(N) \subset p^+ \subset \rr^N$, homeomorphic to $U^N$, which shares the properties of $T(2)$ on $\rr^2$.
\end{prp}

Before we give the proof, we require a technical lemma. Let~$R_M$ denote the $N$-hypercube given by the intersection $(M \cdot p) \cap ((M + 1) \cdot p)$, and define~$R_k$ for $k \geq M$ similarly. We assume here that~$R_M$ is homeomorphic to~$D^N$, the closed disc of $N$ dimensions.
To define~$T(N)$, consider auxiliary lines~$L'_1,L'_2 \in \olines$ which bound each~$R_k$ for $k \geq M$; we make this precise below. The induced subspace will serve as our space~$T(N)$.

To define the relevant notions, we need the following lemma first. Denote the set of vertices of~$R_M$ by $V(R_M)$, and define $V(k \cdot p)$ similarly.

\begin{lem}\label{lem77769}
Let $k \geq M$. Then $V(R_k) \cap V(k \cdot p)$ and $V(R_k) \cap V((k+1) \cdot p)$ are distinct singletons.
\end{lem}

\begin{proof}
	By shifting and shrinking, w.l.o.g., express the set $V(M \cdot p)$ as the set of $\{ 0,2 \}$-words of length~$2^N$, and similarly, $V((M+1) \cdot p)$ as the set of $\{ 1,3 \}$-words of length $2^N$. Now, a vertex $v \in V(M \cdot p)$ is an element of~$(M + 1) \cdot p$ if and only if each of its coordinates lies in the interval~$[1,3]$. Clearly, only the vertex~$(2,2,\ldots,2)$ satisfies this. By an identical argument, the only vertex $v \in V((M + 1) \cdot p)$ in~$M \cdot p$ is~$(1,1,\ldots,1)$.
\end{proof}

For a subset $A \subset \rr^N$, let its interior be denoted by $A^{\circ}$.

\begin{lem}\label{lemPpp}
	There exists a line $L' \in \olines$ such that, for every $k \geq M$ we have $R^{\circ}_k \cap L' \neq \emptyset$.
\end{lem}

\begin{proof}
	Let $v_{k} \in \rr^N$ be as provided by \cref{lem77769}, where
	\[
		v_{k} \in V(R_k) \cap V((k + 1) \cdot p).
	\]
	By definition of $k \cdot p$, there exists $\spec{v} \in V(p)$ for which $v_{k} = (k+1)\spec{v}$. Note that $v_{k}$ is the vertex of $V((k + 1) \cdot p)$ closest to $\origin$. Since the set $\set{v_i}{i < \omega}$ lies on a straight line containing $\origin$, defining~$L' \in \olines$ to be the unique line containing $v_{1}$ suffices.
\end{proof}

As provided by the previous lemma, fix the straight line $L$ that passes through every $R_k$. Let~$B_M$ be the $N$-dimensional open ball contained in $R^\circ_M$ with centre on $L$ and of maximal radius so that $\del{B_M} \cap \del{R_M} \neq \emptyset$. Define a sequence $(B_k)_{k \geq M}$ such that:
\begin{enumerate}
	\item $\del{B_M} \cap \del{R_M} \neq \emptyset$
	\item if $L' \in \olines$ and $L'$ is tangent to $B_M$ then $L'$ is tangent to $B_k$ for every $k \geq M$.
\end{enumerate}
Since all sides of $R_k$ grow by a constant as $k$ increases, a straightforward geometrical argument shows that the latter condition can be satisfied. This also shows that
\[
	\lim_{k \rightarrow \infty} \diam(B_k) =\infty.
\]
The following corollary is now immediate.

\begin{cor}\label{corPpp}
	The sequence $(B_k)_{k \geq M}$ induces an $N$-dimensional hypercone~$\mathcal{H}$. In particular,~$\mathcal{H}$ has infinite $N$-dimensional diameter.
\end{cor}

Since all sides of $(k+1) \cdot p$ are larger than their respective sides of $B_k$, we see that, once~$\mathcal{H}$ enters~$p^+$, it does not escape: the radius of~$B_k$ is bounded above by the shortest side length of~$R_k$, which is bounded above by the shortest side length of~$k \cdot p$. By construction, all bounds are strict.

\begin{proof}[Proof of \cref{geomLem54366}]
With $p \subset \rr^N$ fixed, consider the hypercone $\mathcal{H}$ (induced by the $N$-sphere~$B_M$) as given by lemmas \ref{lemPpp} and \ref{corPpp}. Isolate its \emph{positive filled part}, defined by
	\begin{align*}
		x = (x_1,\ldots,&x_N)  \in \mathcal{H}^+ \iff\\
		&(\exists L' \in \olines)(x \in L' \land (\forall m \leq N)(x_m \geq 0) \land L' \cap \del{B_M} \neq \emptyset).
	\end{align*}
	It is easily seen that
	\[
		\del{\mathcal{H}^+} = \bigcup \set{L \in \olines}{|L \cap \del{B_M}| = 1}
	\]
	and hence $\del{\mathcal{H}^+} \homeo \rr^{N-1}$. And since $\mathcal{H}^+$ is clearly convex, it has trivial homology for every $k > 0$ \cite[10.8]{MR2151660}. By taking the stereographic projection we hence see that $\mathcal{H} \homeo U^N$.
	To define $T(N)$, fix some $\tilde{z} \in L \cap p^+$ (as before, we may assume $\tilde{z} \not\in p$) and take the shift $\mathcal{H}^+ + \tilde{z}$ defined by
\[
	\mathcal{H}^+ + \tilde{z} = \set{x + \tilde{z}}{ x\in \mathcal{H}^+}.	
\]
Since $\mathcal{H}^+ + \tilde{z} \subset p^+$, we define~$T(N) = \mathcal{H}^+ + \tilde{z}$. The important properties of $T(2)$ from lemma \ref{Tprops} carry over to $T(N)$ since $\mathcal{H}^+$ is closed, has infinite $N$-dimensional diameter and $T(N) \homeo U^N$.
\end{proof}

To complete the proof, note that the family~$\Set{t(x)}{x \in \rr^N}$ can be defined as in \cref{dfntxxx}, and shares the properties of \cref{lem:sump}. The map~$J$ (as defined on $\rr^N$) is Borel by the same argument as in \cref{functionLDfn}. Since all remaining lemmas in \cref{sec:trivClosedSet} are algebraic, we conclude:

\begin{thm}\label{thm:myThm2}
	$\hbor\left(\rr^N,G\right) = 0$ for every $N < \omega$ and every group $G$.
\end{thm}

\begin{proof}
	The case $N=1$ is \cite[Theorem 49]{MR1841758}; $N=2$ is \cref{thm:myThm1}. The remaining cases were covered in \cref{sec:final}. Hence the theorem is proved.
\end{proof}

\section{Open Questions}\label{sec:open}

In the present work we have focussed on the special cases of Borel definable group extensions of~$\rr^N$ by a \emph{countable} group $G$. If $G$ is uncountable Borel (recall \cref{dfn:borelDfn}: $G$ is a Borel subset of some standard Borel space, and the group operation is a Borel function), the forcing arguments in \cref{sec:thm1} do not apply as Cohen forcing has the c.c.c.\ \cite{MR895139}.

This obstruction is fundamental \cite[p.\ 265]{MR1841758}: based on an idea of Greg Hjorth, one can construct a
Borel subgroup $G \leq \rr$ and a Borel cocycle~$C$ such that $C$ is a Borel coboundary in $\hbor(\rr,\rr)$
but not in $\hbor(\rr,G)$---by \cref{thm:KRreals} such a $G$ is necessarily uncountable---which shows that \cref{thm:KRreals} fails in general for uncountable $G$. (It
should be noted that a similar construction works for Cantor space~$\cant$ \cite{MR1807391,MR1896675}.) This leaves the following question open: 

\begin{question}
	What can be said about the structure of $\hbor\left(\rr^N,G\right)$ if $G$ is Borel and uncountable and $N \geq 2$?
\end{question}

Contrasting the relationship of \cref{thm:KRreals} with our theorems \ref{thm:myThm1} and \ref{thm:myThm2}, one can consider the impact of the properties of $G$ on $\hbor\left(\rr^N,G\right)$ for varying $N$ and ask:

\begin{question}\label{q66}
	For which (group-theoretic) properties $P$ is the statement
\begin{center}
	$\hbor(\rr,G) = \hbor(\rr^N,G)$ whenever $G$ has property $P$
\end{center}
true for all $N < \omega$? For which properties $P$ do both $\hbor(\rr,G)$ and $\hbor(\rr^N,G)$ vanish for all $N < \omega$ whenever $G$ has property $P$?
\end{question}

For instance, in this paper we have shown that the property of ``being countable'' is sufficient to answer the second part of \cref{q66} affirmatively.

\medskip

While Hjorth and Kanovei-Reeken have proven that $\hbor(\cant,\cant)$ and $\hbor(\cant,G)$ are non-trivial---the latter even for finite $G$ \cite[p.\ 265]{MR1841758}---the following questions remain open to the author's knowledge.

\begin{question}
	What can be said about the structure of $\hbor(\baire,G)$ for countable (or Borel uncountable)~$G$? What about~$\hbor(\baire,\baire)$?
\end{question}
 
Further, one can consider the study of definable group extensions of groups with more ``homological'' structure and solve the associated classification problem in the context of Borel equivalence relations, orbits, and group actions (see \cite{MR1425877,MR2455198} for background). Particular cases in this context have been studied in-depth, and various classification results have been obtained (e.g.\ the preprint \cite{martino2022} for certain countable groups; in spirit, these are related to \cref{q66}).

However, for other spaces answers are scarce. Take for instance the $p$-adic integers $\mathbb{Z}_p$ for $p$ prime. This ring can be expressed as the inverse limit of the sequence $(\zz/p, \zz/p^2, \ldots)$, and is hence Polish \cite{MR4175370}. While progress has been made on related definability problems in terms of homological algebra and Borel equivalence relations (e.g.\ \cite{MRbf,MR4794437}), one can wonder about the effect of the inverse limit construction of $\zz_p$ on trivialising Borel cocycles, and whether these can be turned into Borel coboundaries in a step-by-step argument along the limit. Further, $\zz_p$ also carries algebraic structure (it is a valuation ring, a local ring, etc.), which could be investigated in this context. So, we ask:

\begin{question}
	What can be said about the structure of $\hbor(\mathbb{Z}_p,G)$ for $G$ Borel and countable, or otherwise?
\end{question}

\section{Acknowledgements}
The author's research was funded by the Singapore Ministry of Education through its research grant with number MOE-000538-01.

\bibliographystyle{abbrv}
\bibliography{MASTERBIB.bib}

\begin{thebibliography}{10}

\bibitem{MR1425877}
H.~Becker and A.~S. Kechris.
\newblock {\em The descriptive set theory of {P}olish group actions}, volume
  232 of {\em London Mathematical Society Lecture Note Series}.
\newblock Cambridge University Press, Cambridge, 1996.

\bibitem{MRbf}
J.~Bergfalk, M.~Lupini, and A.~Panagiotopoulos.
\newblock The definable content of homological invariants {I}: \textsf{Ext} and
  $\textsf{lim}^1$.
\newblock {\em Proceedings of the London Mathematical Society}, 129(3):e12631,
  2024.

\bibitem{MR4794437}
J.~Bergfalk, M.~Lupini, and A.~Panagiotopoulos.
\newblock The definable content of homological invariants {II}: {\v{C}}ech
  cohomology and homotopy classification.
\newblock {\em Forum Math. Pi}, 12:Paper No. e12, 2024.

\bibitem{Brattka2021}
V.~Brattka, G.~Gherardi, and A.~Pauly.
\newblock {\em Weihrauch Complexity in Computable Analysis}, pages 367--417.
\newblock Springer International Publishing, Cham, 2021.

\bibitem{MR3381097}
C.~T. Chong and L.~Yu.
\newblock {\em Recursion theory}, volume~8 of {\em De Gruyter Series in Logic
  and its Applications}.
\newblock De Gruyter, Berlin, 2015.
\newblock Computational aspects of definability, With an interview with Gerald
  E. Sacks.

\bibitem{MR1807391}
J.~R.~P. Christensen, V.~Kanovei, and M.~Reeken.
\newblock On {B}orel orderable groups.
\newblock {\em Topology Appl.}, 109(3):285--299, 2001.

\bibitem{MR2151660}
M.~D. Crossley.
\newblock {\em Essential topology}.
\newblock Springer Undergraduate Mathematics Series. Springer-Verlag London,
  Ltd., London, 2005.

\bibitem{dansPaper}
A.~Day, N.~Greenberg, M.~Harrison-Trainor, and D.~Turetsky.
\newblock {Iterated Priority Arguments in Descriptive Set Theory}.
\newblock {\em Bull. Symb. Log.}, 30(2):199--226, 2024.

\bibitem{MR2732288}
R.~G. Downey and D.~R. Hirschfeldt.
\newblock {\em Algorithmic randomness and complexity}.
\newblock Theory and Applications of Computability. Springer, New York, 2010.

\bibitem{MR233919}
A.~M. DuPre, III.
\newblock Real {B}orel cohomology of locally compact groups.
\newblock {\em Trans. Amer. Math. Soc.}, 134:239--260, 1968.

\bibitem{MR4472209}
D.~D. Dzhafarov and C.~Mummert.
\newblock {\em Reverse mathematics---problems, reductions, and proofs}.
\newblock Theory and Applications of Computability. Springer, Cham, \copyright
  2022.

\bibitem{MR3467030}
L.~Fuchs.
\newblock {\em Abelian groups}.
\newblock Springer Monographs in Mathematics. Springer, Cham, 2015.

\bibitem{MR2455198}
S.~Gao.
\newblock {\em Invariant descriptive set theory}, volume 293 of {\em Pure and
  Applied Mathematics (Boca Raton)}.
\newblock CRC Press, Boca Raton, FL, 2009.

\bibitem{MR4175370}
F.~Q. Gouv\^ea.
\newblock {\em {$p$}-adic numbers}.
\newblock Universitext. Springer, Cham, third edition, \copyright 2020.
\newblock An introduction.

\bibitem{hitchcock03}
J.~M. Hitchcock.
\newblock Effective fractal dimension: foundations and applications.
\newblock {\em PhD thesis, \normalfont{Iowa State University, USA}}, 2003.

\bibitem{MR4076}
D.~H. Hyers.
\newblock On the stability of the linear functional equation.
\newblock {\em Proc. Nat. Acad. Sci. U.S.A.}, 27:222--224, 1941.

\bibitem{MR895139}
T.~Jech.
\newblock {\em Multiple forcing}, volume~88 of {\em Cambridge Tracts in
  Mathematics}.
\newblock Cambridge University Press, Cambridge, 1986.

\bibitem{MR1940513}
T.~Jech.
\newblock {\em Set theory}.
\newblock Springer Monographs in Mathematics. Springer-Verlag, Berlin,
  millennium edition, 2003.

\bibitem{MR2731169}
A.~Kanamori.
\newblock {\em The higher infinite}.
\newblock Springer Monographs in Mathematics. Springer-Verlag, Berlin, second
  edition, 2009.
\newblock Large cardinals in set theory from their beginnings, Paperback
  reprint of the 2003 edition.

\bibitem{MR1774679}
V.~Kanovei and M.~Reeken.
\newblock On {B}aire measurable homomorphisms of quotients of the additive
  group of the reals.
\newblock {\em MLQ Math. Log. Q.}, 46(3):377--384, 2000.

\bibitem{MR1896675}
V.~Kanovei and M.~Reeken.
\newblock On {U}lam stability of the real line.
\newblock In {\em Unsolved problems on mathematics for the 21st century}, pages
  169--181. IOS, Amsterdam, 2001.

\bibitem{MR1841758}
V.~Kanove\u{\i} and M.~Reeken.
\newblock On {U}lam's problem concerning the stability of approximate
  homomorphisms.
\newblock {\em Tr. Mat. Inst. Steklova}, 231:249--283, 2000.

\bibitem{MR1321597}
A.~S. Kechris.
\newblock {\em Classical descriptive set theory}, volume 156 of {\em Graduate
  Texts in Mathematics}.
\newblock Springer-Verlag, New York, 1995.

\bibitem{MR756630}
K.~Kunen.
\newblock {\em Set theory}, volume 102 of {\em Studies in Logic and the
  Foundations of Mathematics}.
\newblock North-Holland Publishing Co., Amsterdam, 1983.
\newblock An introduction to independence proofs, Reprint of the 1980 original.

\bibitem{MR2494387}
M.~Li and P.~Vit\'anyi.
\newblock {\em An introduction to {K}olmogorov complexity and its
  applications}.
\newblock Texts in Computer Science. Springer, New York, third edition, 2008.

\bibitem{martino2022}
M.~Lupini.
\newblock The classification problem for extensions of torsion-free abelian
  groups, {I}, 2022, Preprint.
\newblock \textsf{arXiv:2204.05431}.

\bibitem{lutz03}
J.~H. Lutz.
\newblock The dimensions of individual strings and sequences.
\newblock {\em Inform. and Comput.}, 187(1):49--79, 2003.

\bibitem{MR3811993}
J.~H. Lutz and N.~Lutz.
\newblock Algorithmic information, plane {K}akeya sets, and conditional
  dimension.
\newblock {\em ACM Trans. Comput. Theory}, 10(2):Art. 7, 22, 2018.

\bibitem{mayordomo02}
E.~Mayordomo.
\newblock A {K}olmogorov complexity characterization of constructive
  {H}ausdorff dimension.
\newblock {\em Inform. Process. Lett.}, 84(1):1--3, 2002.

\bibitem{MR171880}
C.~C. Moore.
\newblock Extensions and low dimensional cohomology theory of locally compact
  groups. {I}, {II}.
\newblock {\em Trans. Amer. Math. Soc.}, 113:40--63; ibid. 113 (1964), 64--86,
  1964.

\bibitem{MR2526093}
Y.~N. Moschovakis.
\newblock {\em Descriptive set theory}, volume 155 of {\em Mathematical Surveys
  and Monographs}.
\newblock American Mathematical Society, Providence, RI, second edition, 2009.

\bibitem{MR2548883}
A.~Nies.
\newblock {\em Computability and randomness}, volume~51 of {\em Oxford Logic
  Guides}.
\newblock Oxford University Press, Oxford, 2009.

\bibitem{MR507327}
T.~M. Rassias.
\newblock On the stability of the linear mapping in {B}anach spaces.
\newblock {\em Proc. Amer. Math. Soc.}, 72(2):297--300, 1978.

\bibitem{Richter2024}
L.~Richter.
\newblock {On the Definability and Complexity of Sets and Structures}.
\newblock {\em PhD thesis, \normalfont{Victoria University of Wellington, NZ}},
  4 2024.

\bibitem{MR2455920}
J.~J. Rotman.
\newblock {\em An introduction to homological algebra}.
\newblock Universitext. Springer, New York, second edition, 2009.

\bibitem{MR357114}
S.~Shelah.
\newblock Infinite abelian groups, {W}hitehead problem and some constructions.
\newblock {\em Israel J. Math.}, 18:243--256, 1974.

\bibitem{MR1723993}
S.~G. Simpson.
\newblock {\em Subsystems of second order arithmetic}.
\newblock Perspectives in Mathematical Logic. Springer-Verlag, Berlin, 1999.

\bibitem{MR882921}
R.~I. Soare.
\newblock {\em Recursively enumerable sets and degrees}.
\newblock Perspectives in Mathematical Logic. Springer-Verlag, Berlin, 1987.
\newblock A study of computable functions and computably generated sets.

\bibitem{MR120127}
S.~M. Ulam.
\newblock {\em A collection of mathematical problems}, volume no. 8 of {\em
  Interscience Tracts in Pure and Applied Mathematics}.
\newblock Interscience Publishers, New York-London, 1960.

\bibitem{MR280310}
S.~M. Ulam.
\newblock {\em Problems in modern mathematics}.
\newblock Science Editions John Wiley \& Sons, Inc., New York, 1964.

\bibitem{MR1795407}
K.~Weihrauch.
\newblock {\em Computable analysis}.
\newblock Texts in Theoretical Computer Science. An EATCS Series.
  Springer-Verlag, Berlin, 2000.
\newblock An introduction.

\end{thebibliography}

\end{document}